\documentclass[preprint,3p]{elsarticle}

\usepackage{amssymb,comment}

\usepackage{subcaption, graphicx, tikz, color}
\usepackage{todonotes}
\usepackage{hyperref}     
\usepackage{mathtools}
\usepackage{cleveref} 
\usepackage{graphicx} 
\usepackage{newunicodechar}
\usepackage[font=small,labelfont=bf]{caption}

\newcommand{\Ex}{\,{\mathbb{E}}}
\newcommand{\Prob}{{\mathbb{P}}}
\newcommand{\R}{\mathbb{R}}
\newcommand{\Z}{\mathbb{Z}}
\newcommand{\N}{{\mathbb N}}

\renewcommand{\O}{\mathrm{O}}
\newcommand\eins{\mathbf 1} 

\newcommand{\E}[1]{{\mathbb E}\left[#1\right]}				

\newcommand{\p}{{\mathbb P}}
\providecommand{\P}{}
\renewcommand{\P}[1]{{\mathbb P}\left(#1\right)}
\newcommand{\eqdist}{\ensuremath{\stackrel{\mathrm{d}}{=}}}

\newcommand{\convas}{\ensuremath{\stackrel{\mathrm{a.s.}}{\rightarrow}}}
\newcommand{\convl}[1]{\overset{L_{#1}}\longrightarrow}
\DeclareMathOperator{\Var}{Var}
\DeclareMathOperator{\Ber}{Ber}
\DeclarePairedDelimiterXPP\Pk[1]{\p}{(}{)}{}{
  \providecommand\given{\nonscript\:\delimsize\vert\nonscript\:\mathopen{}}  
#1}
\DeclarePairedDelimiterXPP\Pkv[2]{\p_{#1}}{(}{)}{}{
  \providecommand\given{\nonscript\:\delimsize\vert\nonscript\:\mathopen{}}  
#2}
\DeclarePairedDelimiterXPP\Ek[1]{\Ex}[]{}{
  \providecommand\given{\nonscript\:\delimsize\vert\nonscript\:\mathopen{}}  
#1}
\DeclarePairedDelimiterXPP\Ekv[2]{\Ex_{#1}}[]{}{
  \providecommand\given{\nonscript\:\delimsize\vert\nonscript\:\mathopen{}}  
#2}
\DeclarePairedDelimiter\mg{\{}{\}}

\DeclarePairedDelimiter\kl{(}{)}
\DeclarePairedDelimiter\ceil{\lceil}{\rceil}

\DeclarePairedDelimiter{\abs}{\lvert}{\rvert}
\DeclarePairedDelimiterXPP\lnorm[2]{}\lVert\rVert{_{#1}}{#2}
\DeclarePairedDelimiterXPP\lnorml[2]{}\lVert\rVert{_{#1}^{#1}}{#2}

\newcommand\inv{^{-1}}

\newunicodechar{≤}{\le}
\newunicodechar{≥}{\ge}
\newunicodechar{²}{^2}

\newcommand\asec{\alpha^{\mathrm{sec}}}
\newcommand\hopt{h_{\mathrm{opt}}}
\newcommand\lopt{l_{\mathrm{opt}}}
\newcommand\topt{t_{\mathrm{opt}}}

\newcommand\alphan{\rho_n}
\newcommand\indgt{\eins_{G_{t,k}}}
\newcommand\gtk{G_{t,k}}

\newtheorem{thm}{Theorem}
\newtheorem{lem}[thm]{Lemma}
\newdefinition{defi}[thm]{Definition}
\newtheorem{prop}[thm]{Proposition}
\newtheorem{rem}[thm]{Remark}
\newproof{proof}{Proof}
\newproof{pot1}{Proof of Theorem~\ref{thm:moment}}
\newproof{pot2}{Proof of Theorem~\ref{prop:gamma-k}}
\usepackage[normalem]{ulem}

\todostyle{done}{color=green}
\todostyle{ref}{color=cyan, inline}
\begin{document}

\title{Probabilistic analysis of optimal multi-pivot QuickSort}

\author[1]{Cecilia Holmgren\footnote{The research of Cecilia Holmgren was supported by Ragnar Söderberg’s Foundation; the Swedish Research Council; and the Knut and Alice Wallenberg Foundation.}}
\ead{cecilia.holmgren@math.uu.se}

\author[2]{Jasper Ischebeck}
\ead{ischebec@math.uni-frankfurt.de}

\author[3]{Daniel Krenn}
\ead{daniel.krenn@plus.ac.at}

\author[2]{Florian Lesny}
\ead{lesny@math.uni-frankfurt.de}

\author[2]{Ralph Neininger}
\ead{neininger@math.uni-frankfurt.de}
\affiliation[1]{organization={Department of Mathematics, Uppsala University},
            city={Uppsala},
            country={Sweden}}
\affiliation[2]{organization={Institute of Mathematics, Goethe University Frankfurt},
addressline={Robert-Mayer-Straße 10},
postcode={60325},
city={Frankfurt},
country={Germany}}
\affiliation[3]{organization={Fachbereich Mathematik, Paris Lodron Unitversity of Salzburg},
addressline={Hellbrunner Straße 34},
postcode={5020},
city={Salzburg},
country={Austria}}

\begin{abstract}
We consider a multi-pivot QuickSort algorithm using $K\in\N$ pivot elements to partition a nonsorted list into $K+1$ sublists in order to proceed recursively on these sublists.
  For the partitioning stage, various strategies are in use. We focus on the strategy that minimizes the expected number of key comparisons in the standard random model, where the list is given as a uniformly permuted list of distinct elements.

We derive asymptotic expansions for the expectation and variance of the number of key comparisons as well as a limit law for all $K\in\N$, where the convergence holds for all (exponential) moments. For $K\le 4$ we also bound the rate of convergence within the Wasserstein and Kolmogorov--Smirnov distance. 

Our analysis of the expectation is based on classical results for random $m$-ary search trees. For the remaining results, combinatorial considerations are used to make the contraction method applicable.

\end{abstract}

\maketitle



\section{Introduction}

The QuickSort algorithm, introduced by Hoare \cite{hoare:quicksort}, is a divide-and-conquer sorting algorithm. It solves the problem of sorting a list of distinct elements/items from a totally ordered set by selecting an element as \emph{pivot} and partitioning the other elements into two sublists according to whether they are smaller or larger than the pivot and acting recursively on these sublists. QuickSort is commonly used, and variants of it are the standard sorting routines in many programming frameworks such as Oracle's modern Java versions.

Java 7 switched from using the classical QuickSort to a dual-pivot QuickSort introduced by Yaroslavskiy \cite{Yaroslavskiy2009}, which is still used in Java SE as of Java 23 \cite{java23doc}. In general, multi-pivot QuickSort is a variant of QuickSort that uses $K\in \N$ pivots (instead of $K=1$ in the classical QuickSort) to partition into $K+1$ sublists.
 The advantages of multi-pivot QuickSort algorithms include that they reduce the list sizes faster and need fewer passes over the list. In addition, if designed properly, they need fewer key comparisons. The Yaroslavskiy algorithm is analyzed in \cite{Wild2012, wild-nebel-neininger} under the probabilistic model of uniformity of the order of the given items. In the present paper, we focus on the number of key comparisons as a cost measure to analyze the complexity of multi-pivot QuickSort.

\subsection{Multi-pivot QuickSort}

For $K \in \N$, a $K$-pivot QuickSort algorithm sorts a list of more than $K$ elements by
\begin{enumerate}
\item picking $K$ elements as \emph{pivot elements} (according to some
  rule, for example picking the last $K$ elements) and sorting these elements (by some elementary sorting
  algorithm) as $p_1< \dots< p_K$, then
\item partitioning the list into $K+1$ sublists $S_0$, \ldots, $S_K$, each containing only
  the elements below the smallest pivot element $p_1$,
  between two consecutive pivot elements or above the largest pivot element $p_K$,
  and finally
\item proceeding recursively with the $K+1$ new lists.
\end{enumerate}
If the list has $K$ or fewer elements, it uses directly some elementary sorting algorithm. Note that the unspecified components above (how to choose and sort the pivot elements and how the lists of at most $K$ elements are sorted) do not enter our analysis in this paper; our results are valid for any reasonable choice. 

The core of the QuickSort algorithm is formed by the second
step dealing with the dividing (``partitioning'') of the list into $K+1$ sublists, and this  important aspect needs further attention. The order in which an element is compared to the different pivot elements determines how many comparisons are needed. We call these orders and how these are chosen a \emph{partition strategy}. For example, in the case of $K=2$ pivot elements, we might first compare an element with the smaller pivot, $p_1$. If the element is smaller than $p_1$, we can skip the comparison with the larger pivot. 
Aumüller, Dietzfelbinger and Klaue \cite{Aumueller2015, Aumüller2016} developed a partition strategy that is optimal (see \cite{Aumueller-Dietzfelbinger-Heuberger-Krenn:2016:dual-pivot-quick}) with respect to the expected number of comparisons. This means
an \emph{optimal $K$-pivot QuickSort} always chooses the partition strategy that minimizes the expected number of key comparisons. It is based on how many elements it has already sorted into the sublists $S_0$, \ldots, $S_K$ and determines the order in which the next element is compared to the different pivot elements. We will detail later, in Section \ref{sec:partition}, how exactly optimal  $K$-pivot QuickSort determines the optimal partition strategy.

\subsection{Highlights of the results}

In the present paper, we analyze, for a fixed $K\in\N$, the optimal $K$-pivot QuickSort when acting on $n$ elements to be sorted by considering the number $X_n$ of key comparisons. We assume that all $n$ elements are distinct and all permutations of the $n$ elements to be sorted are equally likely. We now highlight our results.
\begin{itemize}
    \item For each $K\in\N$, we obtain the ``shape'' of the asymptotics for the expected number of key comparisons to be $\E{X_n} = \alpha_K n\ln n + \beta_K n + o(n)$ as $n\to\infty$. Here, $\alpha_K$ and $\beta_K$ are suitable constants; in particular,
    $\alpha_K$ can be determined as an expected value related to the optimal partition strategy. Theorem~\ref{thm:moment} provides more details.
    The constants $\alpha_K$ in the main term of the asymptotic formula of $\E{X_n}$ approaches, as the parameter~$K$ is increased, the information-theoretic lower bound $1/{\ln 2}$, to be precise, $\lim_{K\to\infty} \alpha_K = 1/{\ln 2}$; see Theorem~\ref{prop:gamma-k} and the discussion above that theorem.
    \item In Theorem~\ref{thm:variance} we show that for each $K\in\N$ the variance $\Var(X_n)$ is asymptotically $\sigma^2_K n^2$ as $n\to\infty$ for suitable constants~$\sigma_K$. Insights on these constants are given and for $K\le6$ values are provided.
    \item For each $K\in\N$, the normalization $Y_n:=(X_n-\E{X_n})/n$ converges in distribution, as $n\to\infty$, to a $Z_K$ which can be determined from a distributional fixed-point equation. Moreover, for every $\lambda\in\R$, we have $\E{\exp(\lambda Y_n)} \longrightarrow \E{\exp(\lambda Z_K)}<\infty$ as $n\to\infty$. Details are stated in Theorem~\ref{thm_1}. 
    Additionally, Theorem~\ref{thm:smooth} states that the limit $Z_K$ has a smooth Lebesgue density which, together with all its derivatives, is rapidly decreasing.
    Furthermore, we provide asymptotic upper bounds for the rate of convergence of $Y_n$ to $Z_K$, as $n\to\infty$, in the minimal $\ell_p$ metric, $p\ge1$, and in the Kolmogorov--Smirnov metric; see Theorem~\ref{thm:speed} for details.
\end{itemize}

We present the detailed results in Section~\ref{sec:detailed-results}. The remaining paper is organized as follows: Partition strategies are discussed in Section~\ref{sec:partition}; there also related notions needed for the proofs are established. Section~\ref{sec:proofs} contains the actual proofs for all the theorems and other results of Section~\ref{sec:detailed-results}.

\section{Detailed results}
\label{sec:detailed-results}

For expediting our analysis, we assume that the items are numbers from $[0,1]$ being independent and identically, uniformly distributed. Let $p_1$, \ldots, $p_k$ be the last $k$ items%
\footnote{The choice of using the last~$k$ items of the list is for convenience and no restriction. We might choose the pivots independently and uniform from the data to be independent of the data and to address other input models as well.}
of the list and we assume these are sorted, i.e. $0 \le p_1 < \dots < p_k \le 1$. We define the \textit{spacings}
\begin{equation*}\label{def:d}
D_0:=p_1, D_1:=p_2-p_1,\dots, D_K:=1-p_K,
\end{equation*}
and set $D=\left(D_0,\dots,D_K\right)$.
Conditional on the pivots, $D_i$ is the probability that an element is to be classified into the sublist $S_i$. Moreover, still conditional on the pivots, we define the expected number of comparisons with the pivots that an element needs to be classified with the optimal strategy by $\lopt(D)$, see Section~\ref{sec:partition} for details and an extended definition, and we set
\begin{align*}\gamma_K := \E{\lopt(D)}. \end{align*} 

We now present our results on the asymptotic behavior of the number of key comparisons using our multi-pivot QuickSort algorithm.
\begin{thm}\label{thm:moment}
Let $K\in \N$. For the number $X_n$ of key comparisons using optimal $K$-pivot QuickSort we have 
  \begin{equation}
    \E{X_n} = \alpha_K n\ln n + \beta_K n + o(n) \quad (n\to\infty),
  \end{equation}
with a constant $\beta_K \in \R$ and $\alpha_K= \gamma_K/(H_{K+1}-1)$, where $H_n := \sum_{i=1}^n \frac1i$ are the harmonic numbers. In the case of $K≤4$, the error term $o(n)$ can be replaced by $\O(\log n)$.
\end{thm}
In the case of $K \le 4$, there are more explicit expansions of $\E{X_n}$, see \cite{Heuberger-Krenn:2025:multi-pivot-quicksort}, given through
\begin{align}\label{eq:error-bounds}
    \E{X_n}= \alpha_K n \ln(n)+\beta_K n + \delta_K \ln(n)+\epsilon_K+ \O{\left(\frac{1}{n}\right)}.
\end{align}
For $K=3$ we have
\begin{align*}
    \begin{split}
       &\alpha_3 = \frac{133}{78}\\
       &\beta_3=\frac{133}{78}\gamma-\frac{2}{117}\sqrt{3}\pi+\frac{4}{39}\ln(3)+\frac{3}{26} \ln(2)-\frac{6761}{2028},\\
       &\delta_3=\frac{707}{468},\\
       &\epsilon_3=\frac{707}{468}\gamma+\frac{1}{702}\sqrt{3}\pi+\frac{11}{234}\ln(3)+\frac{5}{156}\ln(2)+\frac{70315}{109512}.
    \end{split}
\end{align*}
We will need these more precise expansions of $\E{X_n}$ later only to bound the rate of convergence in Theorem~\ref{thm:speed}. However, we conjecture that such expansions hold for all $K\in \N$ and that the restriction of  $K \le 4$ can be removed in our Theorem~\ref{thm:speed}.

Theorem~\ref{prop:gamma-k} below discusses the parameter $\alpha_K = \lim_{n\to\infty} \frac{\E{X_n}}{n\ln n}$, which is the constant of the leading order in the asymptotic expansion for the expected value $\E{X_n}$ (Theorem~\ref{thm:moment}), in dependence of~$K$.
By increasing $K$, we obtain $K$-pivot QuickSort algorithms where $\alpha_K$ approaches the information-theoretic lower bound $1/{\ln 2}$.
However, the number of partition strategies that need to be considered increases exponentially with $K$, rendering higher values for $K$ practically unusable. One may think about $K$ depending on $n$, however for the scope of this paper, $K$ is a fixed parameter for our QuickSort algorithm.

\begin{thm}\label{prop:gamma-k}
We have 
 $
   \lim_{K\to\infty} \alpha_K = \frac1{\ln 2}.
 $
\end{thm}

We continue with considerations of the second moment.

\begin{thm}\label{thm:variance}
Let $K\in \N$. For the number $X_n$ of key comparisons using optimal $K$-pivot QuickSort we have 
  \begin{equation*}
    \Var(X_n) \sim \sigma^2_K n^2 \quad (n\to\infty),
  \end{equation*}
where
 \begin{equation}\label{sigma-k}
  \sigma^2_K=\frac{K+2}K\; \mathbb{E}\left[\left(\alpha_K \sum_{i=0}^{K} {D_i \ln(D_i)}+
\lopt(D)\right)^2\right].
  \end{equation}
For $K=3$ an explicit value is given as
\begin{equation*}
    \sigma^2_3=\frac{3051169}{657072}-\frac{17689}{36504}\pi^2+\frac{1463}{2808} \ln(2)-\frac{665}{8424}\ln(3) \approx 0.1354.
\end{equation*}
\end{thm}
Note that $\sigma^2_1=(21-2\pi^2)/3$ (an exact expression for $\Var(X_n)$ for $K=1$ is given in Knuth \cite{kn73}), and $\sigma^2_2$ has been derived in \cite{straub}. Numerical values for  $\sigma^2_K$ for $K\in\{4,5,6\}$ are given in Table \ref{table1}.

Next we consider the limiting distribution.

\begin{thm}\label{thm_1}
Let $K\in \N$ and $\lambda\in \R$. For the normalization $Y_n:=(X_n-\E{X_n})/n$ as $n\to\infty$ we have 
\begin{align*}
   Y_n&\stackrel{d}{\longrightarrow} Z_K,\\
   \E{\exp(\lambda Y_n)} &\longrightarrow \E{\exp(\lambda Z_K)}<\infty,
\end{align*}
where the distribution of $Z_K$ is determined as the unique centered, square integrable probability measure such that
\begin{align}\label{fixpoint}
Z_K \overset{\text{d}}{=} \sum_{i=0}^{K} \kl*{D_i Z_K^{(i)}+\alpha_K D_i \ln\left(D_i\right)}+\lopt(D),
\end{align}
where $D,Z_K^{(0)},\dots,Z_K^{(K)}$ are independent and $Z_K^{(i)}$ has the same distribution as $Z_K$ for all $i=0,\dots,K$.
Moreover, we have moments convergence, i.e., all moments of $Y_n$ converge to the moments of $Z_K$.

\end{thm}
\begin{rem} We conjecture that the restriction to square integrability in Theorem~\ref{thm_1} can be weakened to integrability in view of \cite{fija00ecp}.
\end{rem}

\begin{table}
\renewcommand{\arraystretch}{1.5}
\begin{tabular}{c c c l l}
    $K$ & $\gamma_K$ & $\alpha_K$ &  & $\sigma^2_K$\\
    \hline
     1 &1 &2 & $=2$ & 0.420\\
2& $\frac{3}{2}$ &$\frac{9}{5}$ &$=1.8$ & 0.242 \\
3 &$\frac{133}{72}$ &$\frac{133}{78}$ &$\approx1.7051$ & 0.135\\
4 & $\frac{2384}{1125}$ &$\frac{9536}{5775}$ &$\approx1.6513$ & 0.083\\
5 &$\frac{36469}{15552}$& $\frac{182345}{112752}$& $\approx1.6172$ & 0.056 \\
6 & $\frac{31796145419183}{12522149640000}$ &$\frac{31796145419183}{19945995498000}$ & $\approx1.5941$ 
& 0.040
\end{tabular}
\caption{Mean and variance for small $K$. The variance is approximated using a Monte-Carlo method on \eqref{sigma-k} with one million samples of $D$.\label{table1}}
\end{table}
It is well known that limit distributions as in Theorem~\ref{thm_1} have smooth densities, first shown by Fill and Janson \cite{fija00} for the case $K=1$. A criterion of Leckey \cite{leckey19} implies the following.
\begin{thm}\label{thm:smooth}
Let $K\in \N$. The limit $Z_K$ has a smooth Lebesgue density which, together with all its derivatives, is rapidly decreasing, i.e., faster than any polynomial. 
\end{thm}

Finally, we give bounds on the rate of convergence in Theorem~\ref{thm_1} for $K\in\{2,3,4\}$. For the case $K=1$ see \cite{Fill2002}, for $K\ge 5$ we do not know the asymptotic behavior of $\E{X_n}$ well enough. We use the minimal $\ell_p$ metrics for $p \ge 1$ given by
\[
\ell_p\left(\mu,\nu\right):=\inf\left\{\left\Vert X-Y \right\Vert_p: L(X)=\mu, L(Y)=\nu \right\}
\]
on the space of real-valued probability measures with a finite absolute $p$-th moment, where $\left\Vert \cdot \right\Vert_p$ denotes the $L_p$-norm.  Moreover, we use the Kolmogorov--Smirnov metric defined as
\[
\varrho\left(\mu,\nu \right):=\sup_{x \in \R} \abs[\big]{F_\mu(x)-F_\nu(x)},
\]
where $F_\mu$ and $F_\nu$ denote the distribution functions of $\mu$ and $\nu$ respectively. The $\ell_p$ and $\varrho$ distance between two random variables is to be seen as the distance of their probability distributions. We have the following statement.
\begin{thm}\label{thm:speed}
For $K\in\{2,3,4\}$, all $p \ge 1$ and all $\varepsilon>0$ we have, as $n\to\infty$, that
\begin{align*}
\ell_p\left(Y_n,Z_K\right)=\mathrm{O}\left(n^{-1/2+\varepsilon}\right),\qquad 
\varrho\left(Y_n,Z_K\right)=\mathrm{O}\left(n^{-1/2+\varepsilon}\right).
\end{align*}
\end{thm}

\section{Partition strategies and probabilistic set-up}\label{sec:partition}
A partition strategy splits a list of~$n-K$ elements (i.e. excluding the $K$ pivots) into sublists $S_0$, \ldots, $S_K$. It
uses binary search trees with the pivots $p_1$, \ldots, $p_K$ as internal nodes and the lists $S_0$, \ldots, $S_K$ associated to the external nodes. Note that there is only one way of placing the pivots in a given binary search tree as such trees give rise to an ordering. We call these trees \emph{classification trees} and in Figure~\ref{figure:all-comparison-trees-d=3} we see the classification trees for $K=3$ pivot elements.
\begin{figure}
  \centering
  \begin{subfigure}[b]{0.3\columnwidth}
    \centering
    \begin{tikzpicture}[scale=0.6]
      \node [circle,draw] {$p_1$}
      child {node [draw] {$S_0$}
      }
      child {node [circle,draw] {$p_2$}
        child {node [draw] {$S_1$}}
        child {node [circle,draw] {$p_3$}
          child { node [draw] {$S_2$}}
          child { node [draw] {$S_3$}}
        }
      };
    \end{tikzpicture}
    \caption{Classification tree $t_{1}$}
  \end{subfigure}
  \hspace*{1em}
  \begin{subfigure}[b]{0.3\columnwidth}
    \centering
    \begin{tikzpicture}[scale=0.6]
      \node [circle,draw] {$p_1$}
      child {node [draw] {$S_0$}
      }
      child {node [circle,draw] {$p_3$}
        child {node [circle,draw] {$p_2$}
          child { node [draw] {$S_1$}}
          child { node [draw] {$S_2$}}
        }
        child {node [draw] {$S_3$}}
      };
    \end{tikzpicture}
    \caption{Classification tree $t_{2}$}
  \end{subfigure}
  \hspace*{1em}
  \begin{subfigure}[b]{0.3\columnwidth}
    \centering
    \begin{tikzpicture}[level/.style={sibling distance=30mm/#1}, scale=0.6]
      \node [circle,draw] {$p_2$}
      child {node [circle,draw] {$p_1$}
        child { node [draw] {$S_0$}}
        child { node [draw] {$S_1$}}
      }
      child {node [circle,draw] {$p_3$}
        child { node [draw] {$S_2$}}
        child { node [draw] {$S_3$}}
      }
      ;
    \end{tikzpicture}
    \caption{Classification tree $t_3$}
    \label{figure:balanced-comparison-tree-d=3}
  \end{subfigure}
  \begin{subfigure}[b]{0.3\columnwidth}
    \vspace*{5ex} 
    \centering
    \begin{tikzpicture}[scale=0.6]
      \node [circle,draw] {$p_3$}
      child {node [circle,draw] {$p_1$}
        child {node [draw] {$S_0$}}
        child {node [circle,draw] {$p_2$}
          child { node [draw] {$S_1$}}
          child { node [draw] {$S_2$}}
        }
      }
      child {node [draw] {$S_3$}}
      ;
    \end{tikzpicture}
    \caption{Classification tree $t_4$}
  \end{subfigure}
  \hspace*{1em}
  \begin{subfigure}[b]{0.3\columnwidth}
    \centering
    \begin{tikzpicture}[scale=0.6]
      \node [circle,draw] {$p_3$}
      child {node [circle,draw] {$p_2$}
        child {node [circle,draw] {$p_1$}
          child { node [draw] {$S_0$}}
          child { node [draw] {$S_1$}}
        }
        child {node [draw] {$S_2$}}
      }
      child {node [draw] {$S_3$}}
      ;
    \end{tikzpicture}
    \caption{Classification tree $t_5$}
  \end{subfigure}
  \caption{All classification trees for $K=3$. These are all possible binary search trees for $p_1,p_2,p_3$.
  }
  \label{figure:all-comparison-trees-d=3}
\end{figure}

Classifying an element according to a classification tree is done as follows: The element starts at the root node of the tree where it is compared to the pivot in that node. If the element is bigger, it proceeds to the right, otherwise to the left. This decision process is repeated until it reaches an external node with associated sublist~$S_i$ and it is appended to this list~$S_i$. We do this for all $n-K$ elements in the list; in each step we might choose a different classification tree according to our partition strategy. A partition strategy minimizing the expected number of key-comparisons~$\mathbb{E}[X_n]$ is called \textit{optimal}, see~\cite{Heuberger-Krenn:2025:multi-pivot-quicksort}.

We now set-up notions to analyze optimality of partition strategies. Let us fix~$K$. The set of all classification trees for this~$K$ is called $\mathcal{T}$. The number of trees in~$\mathcal{T}$ is given by the Catalan number $\frac1{K+1}\binom{2K}{K}$.
We define the linear function
\begin{equation}\label{def:lt}
    l_t(x_0,\dots,x_K):=\sum_{i=0}^K x_ih_t(i)
\end{equation}
on $\mathbb{R}^{K+1}$ (\cite{Heuberger-Krenn:2025:multi-pivot-quicksort}), where $h_t(i)$ is the depth of external node with sublist $S_i$ in the tree $t\in\mathcal{T}$. Note that $h_t(i)$ is also the number of comparisons needed to classify an element by tree~$t$ into sublist~$S_i$.

We now focus on the process of classifying each of the $n-K$ elements in the list.
Let $C^{(k)}_i$ be the number of elements already classified in the list $S_i$ prior to classifying the $k$-th element; and set $C^{(k)} := \kl{C^{(k)}_0, \dots, C^{(k)}_{K}}$. The work~\cite{Aumüller2016} showed that, given $C^{(k)}$, the probability that the $k$-th element is classified in the sublist $S_i$ is proportional to $C^{(k)}_i+1$ (see \cite{Aumueller2015} for the special case of $K=2$). Therefore, the optimal $K$-pivot QuickSort algorithm chooses the classification tree $t\in\mathcal{T}$ that minimizes $l_t(C^{(k)}+1)$. In case of a tie, we use an arbitrary but fixed order on $\mathcal{T}$ and choose the smallest of those trees. 
We write $\gtk$ for the event that the $k$-th element is classified according to a tree $t$.

Let us now focus on the limiting behavior. Given $D$ (see beginning of Section~\ref{sec:detailed-results}), the probability for an element to go into sublist $S_i$ is given by $D_i$. Therefore, $l_t(D)$ is the expected number of comparisons for classifying an element. We define the classification tree $\topt$ as the tree~$t\in\mathcal{T}$ that minimizes $l_t(D)$. (We denoted that quantity by $\lopt(D)$ in Section~\ref{sec:detailed-results}). Note that there is almost surely only a single minimizer $\topt$. We also write $\hopt$ for $h_{\topt}$.
Furthermore, for a tree $t \in \mathcal{T}$, we define its \textit{asymptotic cone} as
\[
    C_t:=\mg[\big]{\left(x_0,\dots,x_K\right)\in \mathbb{R}_{\ge 0}^{K+1} : l_t\left(x_0,\dots,x_K\right) \le l_{t'}\left(x_0,\dots,x_K\right) \text{ for all } t' \in \mathcal T}.
\]
Note that by definition, $D \in C_{\topt}$.

As for the $K$-pivot QuickSort algorithm, the true probabilities $D$ are not known when classifying the $k$-th element and so the algorithm approximates these probabilities using $C^{(k)}/k$, quantities that converges to~$D$. Therefore, as the algorithm classifies more elements, it updates its strategy to reflect changes in $C^{(k)}/k$.

Let us come back to the probabilistic part of the analysis.
The number of key comparisons to classify each of the $n-K$ elements of a list into $K+1$ sublists is denoted by $P_n$. We may write this quantity as
\begin{equation}\label{def:pn}
    P_n = \sum_{k=1}^{n-K} \sum_{t\in\mathcal T}\sum_{i=0}^K\indgt\eins_{\mg{a_k\in S_i}} h_t(i) + R_K,
\end{equation}
where $a_1$, \ldots, $a_n$ denote the elements of the unsorted list. The term $R_K$ is the number of comparisons required to sort the $K$ pivots, and, by choosing an appropriate sorting algorithm, we can ensure $R_K \le K^2$. In particular, for fixed $K$ we have $R_K = \O(1)$ as $n\to\infty$.

Writing $I_i^{(n)} := \abs{S_i}$ for the sublist sizes, $X_n$ can thus be defined by the recursion
\begin{equation}\label{x-recursion}
    X_n \overset d= \sum_{i=0}^K X^{(i)}_{I_i^{(n)}} + P_n
\end{equation}
for $n≥K$,
where $X^{(i)}_k$ for $k\in \N$ are copies of $X_k$ that are independent of each other and of the $(I_0^{(n)}, \dots, I_K^{(n)},P_n)$, $n≥K$.

\section{Proofs}
\label{sec:proofs}

The proofs of our results are organized as follows. In Section \ref{proof_exp} we show the expansion of the mean of $X_n$ from \Cref{thm:moment} with the help of classical results on the $m$-ary search trees. Theorem~\ref{thm_1} is shown by an application of the contraction method in Section \ref{sec:proof_thm1}. The framework of the contraction method used is recalled in Section \ref{sec:proof_thm1}. Theorem~\ref{thm:variance} then follows as a corollary to Theorem~\ref{thm_1}. The properties of the densities of the limit distributions of Theorem~\ref{thm:smooth} are shown in Section \ref{sec:smooth}. The bounds on the rate of convergence from Theorem~\ref{thm:speed} are finally shown in Section \ref{sec:rate}. Here, we need to develop a couple of new estimates, since the problem for $K\in\{2,3,4\}$ is more involved compared to the case $K=1$.
\subsection{Asymptotic expansion of \texorpdfstring{$\E{X_n}$}{the first moment}} \label{proof_exp}
If the algorithm knew $D$ beforehand, it could always use the asymptotically optimal classification tree $\topt$ where $l_t(D)$ is minimal, see \eqref{def:lt}.
This ``oracle strategy'' uses only one tree and is easier to analyze. However, we show, see Lemma~\ref{lem:nah-an-opt}, that the difference
between the actual strategy and the oracle strategy in the number of comparisons is minor. Furthermore, the number of comparisons when using $\topt$ is closely related to the expected internal path length of a $K+1$-ary search tree, which shows the asymptotics of $\E{X_n}$ in \Cref{thm:moment}.


\begin{lem}\label{lem:nah-an-opt}
    The additional cost of the algorithm compared to always using $\topt$ is of sub-polynomial order:
    \begin{equation*}
    \alphan := \Ek{P_n} - \Ek*{\sum_{i=0}^K I^{(i)}_n \hopt(i)} = \O\kl*{(\log n)^2}
    \end{equation*}
\end{lem}

\begin{proof}
    The sum $\sum_{i=0}^K I^{(i)}_n\hopt(i)$ is the number of comparisons of elements with pivots if we were always using the optimal classification tree.
    Condition on $D$. Let $t_k$ be the classification tree used for the $k$-th element and remember $\topt$ being the optimal classification tree. For every classification tree $t\ne \topt$, consider
    \begin{equation*}
        \psi_t(x_0, \dots, x_K) := \sum_{i=0}^K x_i (h_t(i) - \hopt(i)),
    \end{equation*}
    so that $\psi_t(D)$ is
    the expected number of additional comparisons sorting a single element using $t$ instead of $\topt$ conditional on $D$. By \eqref{def:pn}, using the law of total probability,
    \begin{align}
        \alphan &= \sum_{k=1}^{n-K} \sum_{i=0}^K 
        \E{
        \sum_{t\in\mathcal T}
            \indgt \eins_{\mg{a_k \in S_i}} h_t(i) - \eins_{\mg{a_k\in S_i}}\hopt(i)
        } + \E{R_n}
        \notag\\&= \sum_{k=1}^{n-K} 
        \sum_{t\in\mathcal T}
        \E{
            \indgt
            \sum_{i=0}^K
            \eins_{\mg{a_k \in S_i}} (h_t(i) - \hopt(i))
        } + \E{R_n}
        \notag\\&=\label{alpha-sum}
        \sum_{k=1}^{n-K} \sum_{t\in\mathcal T} \Ek[\big]{\Pk*{\gtk \given D} \psi_t(D)} + \O(1).
    \end{align}
    Because the algorithm always chooses the best classification tree, the algorithm chooses $\topt$ over a strategy $t$ if and only if $l_t(C^{(k)} + 1)$ is bigger than $l_{\topt}(C^{(k)} + 1)$ or if there is a tie that is resolved in favor of $\topt$. We can model this tiebreaking by adding either $\frac12$ or $-\frac12$ to the inequality, so that $\topt$ is chosen over $t$ if and only if $\psi_t\kl*{C^{(k)}} = l_t(C^{(k)})-l_{\topt}(C^{(k)}) > l_{\topt}(1) - l_t(1)\pm \frac12 =: b_t$. This $b_t$ is some constant in $\R\setminus \Z$.
    Furthermore, $\topt$ is optimal for $D$, so $\psi_t(D) ≥ 0$ for all $t\ne\topt$, and we can bound $\Pk*{\gtk}$ by $\Pk*{\psi_t\kl*{C^{(k)}}<b_t}$.

    For fixed $t$ and conditional on $D$, the variables $\psi_t(C^{(k)})$ are linear combinations of independent indicator variables with $\Ek{\psi_t(C^{(k)})\given D} = (k-1)\psi_t(D)$ and coefficients bound by $K$, and thus Hoeffding's inequality can be applied to yield
    \begin{align*}
        \Pk*{\psi_t\kl*{C^{(k)}} < b_t \given D} &≤ \exp\kl*{-\frac{2((k-1)\psi_t(D)-b_t)^2}{4K²(k-1)}}
        \\&= \exp\kl*{-\frac{k-1}{2K²}\psi_t^2(D) + \O\kl*{\psi_t(D) + (k-1)\inv}}.
    \end{align*}
    Taking the expectation over $D$,
    \[
        \Ek[\big]{\Pk*{\gtk \given D} \psi_t(D)}
        ≤ \Ek*{\psi_t(D)\exp\kl*{-\frac{k-1}{2K²}\psi_t^2(D) + \O\kl*{\psi_t(D) + (k-1)\inv}}}.
    \]
    Now decompose the expectation on whether $\psi_t^2(D)$ is smaller than ${4K²\ln k}/\kl{k-1}$ into
    \begin{align*}
        \Ek[\big]{\Pk*{\gtk \given D} \psi_t(D)} ≤ K\sqrt{\frac{4\ln k}{k-1}}\Pk*{\psi_t(D)\le K\sqrt{\frac{4\ln(k)}{k-1}}} + \exp\kl*{-2\ln k+\O(1)}.
    \end{align*} Since $\psi_t(D)$ is a linear combination of $D$ and $D$ has bounded density (on a simplex), $\psi_t(D)$ itself has bounded density (on $\topt\ne t$), so
    \begin{equation*}
        \Ek[\big]{\Pk*{\gtk \given D} \psi_t(D)}
        = \O\kl*{\frac{\log k}k} + \O\kl*{k^{-2}}
    \end{equation*}
    and, considering \eqref{alpha-sum}, $\alphan$ is of order $\O\kl*{(\log n)²}$.
\end{proof}

\begin{lem}\label{lem:opt-is-path-length}
    Define the sequence $(\Psi_n)_{n≥0}$ recursively as
    \begin{equation}\label{psi-recursion}
        \Psi_n := \Ek*{\sum_{i=0}^K \Psi_{I^{(i)}_n} + I^{(i)}_n\hopt(i)},
    \end{equation}
    with $\Psi_n := 0$ for $n≤ K$.
    Then, there is a constant $\widehat \beta_K\in \R$, such that
    \begin{equation*}
        \Psi_n
        = \frac{\gamma_K}{H_{K+1}-1} n\ln n + \widehat \beta_K n + o(n).
    \end{equation*}
\end{lem}

\begin{proof}
    Assume $n>K$.
    Conditional on $D$, we have
    $
        \Ek*{I^{(i)}_n\given D} = \kl*{n - K} D_i,
    $
    so
    \begin{align*}
        \Ek*{\sum_{i=0}^K I^{(i)}_n\hopt(i)}
        &= \sum_{i=0}^K \Ek*{\Ek*{I^{(i)}_n\hopt(i) \given D}}
        = \kl*{n - K}\sum_{i=0}^K \Ek*{ D_i\hopt(i)}
        = \gamma_K \kl*{n - K}.
        \label{l2:pf:1}
    \end{align*}
    Therefore, $\frac{\Psi_n}{\gamma_K}$ satisfies the recursion
    \begin{equation*}
        \frac{\Psi_n}{\gamma_K} = (n-K) + \E{\sum_{i=0}^K \frac{\Psi_{I^{(i)}_n}}{\gamma_K}},
    \end{equation*} which is the same recursion as for the expected internal path length of a $K+1$-ary tree. This expectation has been identified by Mahmoud \cite[middle of page 112]{mahmoud:internal-path-length}, where he explicitly derives
    \begin{equation*}
        \frac{\Psi_n}{\gamma_K} = \frac1{H_{K+1}-1} n\ln n + \frac{\widehat \beta_K}{\gamma_K} n + o(n)
    \end{equation*}
    for some explicitly given $\widehat \beta_K\in \R$, see also \cite{holmgren:path-length,mun11}.
\end{proof}

\begin{pot1}
    Define $\Upsilon_n := \E{X_n} - \Psi_n$. Both $\E{X_n}$ with \eqref{x-recursion} and $\Psi_n$ with \eqref{psi-recursion} fulfill similar recursions, so remembering $\alphan$ from Lemma~\ref{lem:nah-an-opt} we obtain
    \begin{equation}\label{upsilon-eq}
        \Upsilon_n = \alphan + \E{\sum_{i=0}^K \Upsilon_{I^{(i)}_n}}.
    \end{equation}
    By Lemma \ref{lem:opt-is-path-length},
    \[
        \E{X_n} = \Psi_n + \Upsilon_n = \frac{\gamma_K}{H_{K+1}-1} n\ln n + \widehat \beta_K n + o(n) + \Upsilon_n,
    \]
    so it suffices to show that $\Upsilon_n \sim c_\alpha n$ for some $c_\alpha$.
    
    Recursions for $m$-ary search trees such as \eqref{upsilon-eq} with small $\alphan$ were studied by Chern and Hwang~\cite{hwang:mnary}. Instead of our probabilistic notation, they prefer to write \eqref{upsilon-eq} as a sum: Note that $I^{(i)}_n$ is distributed as $I^{(1)}_n$, and $I^{(1)}_n = j$ if and only if the $j+1$-th-smallest element is the smallest pivot. The other $K-1$ pivots can be chosen out of $n-1-j$ elements, so we obtain that for $j≤n-K$
    \begin{equation}
        \P{I^{(i)}_n=j} = \frac{1}{\binom{n}{K}}\binom{n-1-j}{K-1}
        \quad\text{ and thus }\quad \Upsilon_n = \alphan + \frac {K+1}{\binom nK}\sum_{j=0}^{n-K}\binom{n-1-j}{K-1}\Upsilon_j.
        \label{upsilon-eq-sum}
    \end{equation}
    In the notation of \cite{hwang:mnary}, the terms $m, t, \tau, b_n$
    correspond in our setting to $K+1,1,K, \alphan$, respectively, and so \eqref{upsilon-eq-sum} is a special case of the recursion (5) in \cite{hwang:mnary}($\tau$ is misspelled as $r$ there). In \cite[Proposition~7]{hwang:mnary}, it is shown for this recursion that
    \[
        \Upsilon_n \sim \frac{K_\alpha}{H_{K+1}-1}n,
    \]
    where
    \[
        K_\alpha := \sum_{j=0}^\infty \frac{\alpha_j}{(j+1)(j+2)},
    \]
    as long as $\alpha_n = o(n)$ and $\sum_{j} \alpha_j j^{-2} <\infty$, which we show in Lemma~\ref{lem:nah-an-opt}. See also Fill and Kapur \cite{fika05} for transfer theorems for $m$-ary search trees.
\end{pot1}

We finish this section with the proof of Theorem~\ref{prop:gamma-k}, which shows that the algorithm approaches the optimal first-order term for $K\to\infty$.

\begin{pot2}
    
    When the classification tree is complete (that is, the height of the leaves differs by at most 1), the algorithm needs at most $\ceil{\log_2(K)}$ comparisons, so $\gamma_K ≤ \log_2(K) + 1$.
    Since the expected number of comparisons is bounded from below by the expected binary entropy
    \begin{equation*}
        \gamma_K ≥ \E{\sum_{i=0}^K D_i \log_2(D_i)} = \frac{H_{K+1}-1}{\ln(2)} \sim \log_2(K),
    \end{equation*}
    this bound is sharp.
    
\end{pot2}

\subsection{The contraction method}\label{sec:cm}
To show the convergence results, we will use the contraction method in the form of \cite{Wild2012}; for a more general introduction also see \cite{Roesler2001} and \cite{Neininger2001}. Let $\left(X_n\right)_{n \ge 0}$ denote a sequence of real-valued random variables satisfying the distributional recurrence
\begin{align*}
    X_n \overset{d}{=} \sum_{i=0}^K A_i^{(n)} X_{I_i^{(n)}}^{(i)}+b_n
\end{align*}
for $n \ge n_0$, where $\left(X_n^{(0)}\right)_{n \ge 0},\dots,\left(X_n^{(K)}\right)_{n \ge 0}$ and $\left(A_0^{(n)},\dots,A_K^{(n)},b_n\right)$ are independent and $X_j^{(i)}$ is distributed as $X_j$ for all $i=0,\dots,K$ and $j \ge 0$. The coefficients $A_i^{(n)}$ and $b_n$ are real random variables and $I^{(n)}=\left(I_0^{(n)},\dots,I_K^{(n)}\right)$ is a vector of random integers in $0,\dots,n-K$, while $K$ and $n_0$ are fixed numbers. Furthermore, we assume that the coefficients are square-integrable and the following conditions hold:
\begin{enumerate}
    \item[(A)] $\left(A_0^{(n)},\dots,A_K^{(n)},b_n \right) \overset{\ell_2}{\to} \left(A_0,\dots,A_K,b\right)$,
    \item[(B)] $\sum_{i=0}^K {\E{A_i^2}}<1$,
    \item[(C)] $\sum_{i=0}^K {\E{\eins_{\{I_i^{(n)}\le k\}\cup\{I_i^{(n)}=n\}} \bigl(b_i^{(n)}\bigr)^2}} \to 0$ as $n \to \infty$ for all constants $k \ge 0$.
\end{enumerate}
Then we have $X_n \overset{\ell_2}{\to} X$, where $X$ is the unique fixed point among all centered random variables with finite second moments of
\begin{align*}
    X \overset{d}{=} \sum_{i=0}^K A_i X^{(i)}+b,
\end{align*}
where $\left(A_0,\dots,A_K,b\right),X^{(0)},\dots,X^{(K)}$ are independent and $X^{(i)}$ is distributed as $X$ for $i=0,\dots,K$.
\subsection{Proof of Theorem~\ref{thm_1}} \label{sec:proof_thm1}
The normalized number of key comparisons $Y_n$ satisfy the recurrence
\begin{align*}
Y_n \eqdist \sum_{i=0}^{K} {\frac{I_i^{(n)}}{n}Y_{I_i^{(n)}}^{(i)}}+\frac{1}{n}\left(P_n-\mathbb{E}[X_n]+\sum_{i=0}^K{\mathbb{E}\left[X_{I_i^{(n)}}\mid I_i^{(n)}\right]}\right).
\end{align*}
To use the contraction method, we have to show that the following conditions hold: 
\begin{enumerate}[1)]
    \item $\frac{I_i^{(n)}}{n} \overset{L_2}{\longrightarrow} D_i$ for $i=0,\dots,K$,
    \item $\frac{P_n}{n} \overset{L_2}{\longrightarrow} \sum_{t \in \mathcal{T}} \eins_{\left\{D \in C_t\right\}} l_t(D)$,
    \item $\frac{1}{n}\left(\sum_{i=0}^K{\mathbb{E}\left[X_{I_i^{(n)}}\mid I_i^{(n)}\right]}-\mathbb{E}[X_n]\right) \overset{L_2}{\longrightarrow}
    \sum_{i=0}^K {\alpha_K D_i \ln(D_i)}$,
    \item $\sum_{i=0}^{K} {\mathbb{E}\left[D_i^2\right]} < 1$,
    \item $\E{\eins_{\left\{I_i^{(n)} \le k\right\}\cup \left\{I_i^{(n)}=n\right\}} \left(\frac{I_i^{(n)}}{n}\right)^2 } \overset{n \to \infty}{\longrightarrow} 0$ for all $k \in \mathbb{N}$ and $i=0,\dots,K$.
\end{enumerate}
Given $D=(d_0,\dots,d_K)$, $I_i^{(n)}$ is multinomially $M(n-K;d_0,\dots,d_K)$ distributed. The strong law of large numbers gives us the almost sure convergence of $\frac{I_i^{(n)}}{n}$ towards $D_i$, and the dominated convergence theorem yields the convergence in $L_2$. Along these lines, the fifth condition also follows. We will now show condition 2).
\begin{lem}
We have
\[
\frac{P_n}{n} \overset{L_2}{\longrightarrow} \sum_{t \in \mathcal{T}} {\eins_{\left\{D \in C_t\right\}}l_t(D)}.
\]
\end{lem}
\begin{proof}
Recall that $a_1,\dots,a_n$ are the elements of the unsorted list, $\gtk$ is the event that $a_k$ is sorted with the classification tree $t \in \mathcal T$ and $\left\{a_k \in S_j\right\}$ denotes that $a_k$ is sorted into sublist $S_j$. We now define for $t \in \mathcal{T}$ and $0 \le j \le K$ the random variable
\[
A_{t,j}^{(n)}:=\sum_{k=1}^{n-K} \indgt \eins_{\mg{a_k \in S_j}}.
\]
We claim $\frac{A_{t,j}^{(n)}}{n} \overset{L_2}{\longrightarrow} \eins_{\left\{D \in C_t\right\}} D_j$. In this paper, this is shown exemplarily for $K=3$ and $t=t_1$, but it works analogously for other trees, since a tree $t \in \mathcal{T}$ is chosen as the classification tree if and only if a set of inequalities in $\left(x_0,\dots,x_K\right)$ is fulfilled.

To show that, we define random walks $W_1=\left(W_{1,i}\right)_{i \ge 0}$ and $W_2=\left(W_{2,i}\right)_{i \ge 0}$ 
by
\begin{align*}
    W_{1,i}&:=\sum_{m=1}^{i} {\eins_{\left\{a_m \in S_1\right\}}-\eins_{\left\{a_m \in S_3\right\}}},
\\
    W_{2,i}&:=\sum_{m=1}^{i} {\eins_{\left\{a_m \in S_0\right\}}-\eins_{\left\{a_m \in S_2\right\}}-\eins_{\left\{a_m \in S_3\right\}}}.
\end{align*}
If $W_1$ is nonnegative, the algorithm chooses $t_1$ over $t_2$ and if $W_2$ is positive, it chooses $t_1$ over $t_3$.

Conditionally on $D=(d_0,\dots,d_3)$, the processes $W_1$ and $W_2$ are two simple walks on $\mathbb{Z}$ with constant probabilities to go one step up, one step down, or stay in the actual state. If $d_1>d_3$ and $d_0>d_2+d_3$, $W_1$ and $W_2$ tend to infinity by the strong law of large numbers. Thus, there exists a random $n_0 \in \mathbb{N}$ such that both random walks are positive for every $i \ge n_0$, so the random walk $W:=\min\{W_1,W_2\}$ is also positive for every $i \ge n_0$. This implies that starting from index $n_0$, the classification tree $t_1$ is always used. With $\left\vert S_j \right\vert = I_j^{(n)}$ we get
\[
\frac{I_j^{(n)}-n_0}{n} \le \frac{A_{t_1,j}^{(n)}}{n} \le \frac{I_j^{(n)}}{n}
\]
and therefore on $\left\{D_1>D_3,D_0>D_2+D_3\right\} = \mg{D\in C_{t_1}}$  we have $\frac{A_{t_1,j}^{(n)}}{n} \to D_j$ almost surely. Similarly, we can conclude $\frac{A_{t_1,j}^{(n)}}{n} \to 0$ almost surely on the complement, because $W_1$ or $W_2$ tends to $-\infty$ almost surely. Using the dominated convergence theorem, we find
\[
\frac{A_{t_1,j}^{(n)}}{n} \overset{L_2}{\longrightarrow} \eins_{\left\{D \in C_{t_1}\right\}} D_j.
\]
We now use the fact that
\[
P_n=\sum_{t \in \mathcal{T}} \sum_{j=0}^{K} {A_{t,j}^{(n)}h_j(t)}+R_K
\]
with $R_K=\O(1)$.
So
\begin{equation*}
   \frac{P_n}{n} \convl2 \sum_{t\in\mathcal T}\sum_{j=0}^K \eins_{\mg{D\in C_t}} D_j h_j(t) = \sum_{t\in\mathcal T}\eins_{\mg{D\in C_t}} l_t(D),
\end{equation*}
which concludes the proof.
\end{proof}

With the expansion $\E{X_n} = \alpha_K n\ln n + \beta_K n + o(n)$ standard calculations imply
\begin{equation*}
    \begin{split}
    \frac{1}{n} \left(\sum_{i=0}^{K} {\mathbb{E}\left[X_{I_i^{(n)}}\mid I_i^{(n)}\right]}-\mathbb{E}[X_n]\right) = \alpha_K \sum_{i=0}^K \frac{I_i^{(n)}}{n}\ln\frac{I_i^{(n)}}{n}+o(1).
\end{split}
\end{equation*}
The continuous mapping theorem now yields $\frac{I_i^{(n)}}{n}\ln\frac{I_i^{(n)}}{n} \to D_i \ln(D_i)$ almost surely and with the dominated convergence theorem also in $L_2$. This proves condition 3).

The spacings $D_i$, $i=0,\dots,K$, are identically beta$(1,K)$-distributed since $D_0$ is the minimum of $K$ independent, uniformly on $[0,1]$ distributed random variables. Therefore, condition 4) also holds and the first part of Theorem~\ref{thm_1} follows with the contraction method. 

For the second part of \Cref{thm_1} we use the following proposition, which is a straightforward extension of an argument of R\"osler \cite[Section 4]{Roesler1991}; see also Fill and Janson \cite{Fill2002} for a quantified extension of \cite[Section 4]{Roesler1991} and \cite[Lemma 4.3]{neru99}.
\begin{prop}\label{prop_1}
    If conditions (A)-(C) in Section 2.2 as well as 
    \begin{enumerate}
        \item[(a)] $\sup_{n \in \mathbb{N}} \left\Vert b_n \right\Vert_{\infty}<\infty$,
        \item[(b)] $\forall n \in \mathbb{N}\colon \sum_{i=0}^K {\left(A_i^{(n)}\right)^2}<1$
    \end{enumerate}
    hold, we have for all $\lambda \in \mathbb{R}$
    \[
    \mathbb{E}\left[\exp(\lambda X_n)\right] \to \mathbb{E}\left[\exp(\lambda X)\right]<\infty.
    \]
\end{prop}
The second part of \Cref{thm_1} follows from Proposition \ref{prop_1}; cf.~also \cite[Theorem 5.1]{neru99}. Note that the argument there also implies the uniform integrability of the $Y_n$. Convergence of the Laplace transform and uniform integrability imply the convergence of all moments claimed in \Cref{thm_1}.
Since the convergence holds with all moments, we directly get $\Var(X_n) \sim \sigma_K^2n$ for $n \to \infty$, where $\sigma_K^2=\E{Z_K^2}$. Squaring \eqref{fixpoint} and using $\E{Z_K}=0$ gives us 
\begin{align*}
    \E{Z_K^2}=&\E{\sum_{i=0}^K \left(D_i Z_K^{(i)}\right)^2}+\E{\left(\alpha_K\sum_{i=0}^K D_i \ln(D_i)+\lopt(D)\right)^2}\\
    &+2\E{\left(\sum_{i=0}^K D_i Z_K^{(i)}\right)\left(\alpha_K\sum_{i=0}^K D_i \ln(D_i)+\lopt(D)\right)}.
\end{align*}
The last summand vanishes since $D$ and $Z_K^{(i)}$ are independent. Using once again independence of $D$ and $Z_K^{(i)}$ for the first term yields us
\begin{align*}
    \left(1-\sum_{i=0}^K \E{D_i^2}\right)\E{Z_K^2}=\E{\left(\alpha_K\sum_{i=0}^K D_i \ln(D_i)+\lopt(D)\right)^2}.
\end{align*}
Theorem~\ref{thm:variance} follows with $\sum_{i=0}^K \E{D_i^2}=\frac{2}{K+2}$.
\subsection{Proof of Theorem~\ref{thm:smooth}} \label{sec:smooth}
Leckey \cite{leckey19}, building on \cite{fija00}, shows that $Y$ satisfying the recurrence $Y \overset{d}{=} \sum_{i=0}^{\infty} A_i Y^{(i)}+b$ has a smooth and bounded density function if the following conditions hold, where $\alpha^{\max}$ is the largest element and $\asec$ the second largest element in $\left(A_i\right)_{i \ge 0}$:
\begin{enumerate}
\item There exists a constant $a>0$, such that $\mathbb{P}\left(\alpha^{\max} \ge a\right)=1$,
\item there are constants $\tau$, $\nu>0$, such that $
\P{\asec \le x } ≤ \tau x^{\nu}$  for all $x>0$,
\item $\mathbb{P}\left(A_i \le 1 \right)=1$,
\item there does not exist a $c \in \mathbb{R}$ such that $\mathbb{P}\left(Y=c\right)=1$,
\item $\mathbb{P}\left(\sum_{i=0}^{\infty}{\eins_{\left\{A_i \in (0,1)\right\}} \ge 1}\right)>0$.
\end{enumerate}
In our case, $A_i=D_i$ for $i=0,\dots,K$ and $A_i=0$ for $i>K$. If we choose $a:=1/(K+1)$, condition 1 holds. Conditions 3 and 5 hold since $D_i \in (0,1)$ almost surely. Because $Y$ has a positive variance by \Cref{thm:variance}, there could not be a $c \in \mathbb{R}$ with $Y=c$ almost surely. For condition 2, we need a little calculation:\\
Since $D_0+\dots+D_K=1$, for $x \in \left(0,1/K\right)$ we have
\begin{equation*}
    \begin{split}
    \P{\asec ≤ x} & \le \P{\max\mg*{D_0,\dots,D_K} \ge 1-K \cdot x}\\
    & = \mathbb{P}\left(\bigcup_{i=0}^{K} \left\{D_i\ge 1-K \cdot x\right\}\right)\\
    & \le (K+1)\mathbb{P}\left(D_0 \ge 1-K \cdot x\right)\\
    & = (K+1) \mathbb{P}\left(\min\left\{U_1,\dots,U_K\right\} \ge 1-K \cdot x \right)\\
    & = (K+1) \left(K \cdot x \right)^K=(K+1)K^Kx^K.
\end{split}
\end{equation*}
If we choose $\tau:=\left(K+1\right)K^K$ and $\nu=K$, condition 2 holds for $x\in \left(0,\frac{1}{K}\right)$. The function $g(x):=(K+1)K^Kx^K$ fulfills $g\left(1/K\right)=K+1>1$ and increases monotonically on $(0,1)$. Therefore, condition 2 also holds for $x\ge 1/K$.

\subsection{Rate of Convergence} \label{sec:rate}
In the present section, we are proving Theorem~\ref{thm:speed}. We start with the bounds of the speed of convergence in the $\ell_p$ metrics. For later use, we have the following technical result.
\begin{lem}\label{tech}
    For $K \in \mathbb{N}$ and all $\varepsilon>0$, there exists a $\xi>1$ such that
    \[
    \sum_{j=0}^{n-K} \frac{\left(n-j-1 \right)!}{\left(n-K-j \right)!}j^{1+\varepsilon} \le \frac{n^{\varepsilon}}{\xi} \frac{n!}{K(K+1)(n-K-1)!}.
    \]
\end{lem}
\begin{proof}
    For a fixed $\varepsilon>0$ we bound
    \begin{equation*}
        \begin{split}
            \sum_{j=0}^{n-K} \frac{\left(n-j-1 \right)!}{\left(n-K-j \right)!}j^{1+\varepsilon} & = \sum_{j=0}^{\left\lfloor \frac{n}{2} \right\rfloor} \frac{\left(n-j-1 \right)!}{\left(n-K-j \right)!}j^{1+\varepsilon} +\sum_{j=\left\lfloor \frac{n}{2} \right\rfloor+1}^{n-K} \frac{\left(n-j-1 \right)!}{\left(n-K-j \right)!}j^{1+\varepsilon}\\
            & \le  n^\varepsilon \left(\frac{1}{2^\varepsilon}\sum_{j=0}^{\left\lfloor \frac{n}{2} \right\rfloor} \frac{\left(n-j-1 \right)!}{\left(n-K-j \right)!}j +\sum_{j=\left\lfloor \frac{n}{2} \right\rfloor+1}^{n-K} \frac{\left(n-j-1 \right)!}{\left(n-K-j \right)!}j \right).
        \end{split}
    \end{equation*}
    The first summand in the latter display contributes at least $\frac{1}{5}$ of the initial sum, while the second part is smaller than $\frac{4}{5}$ of the initial sum.\\
    (The case $K=1$ follows from the Gaussian sum formula, while in the case $K \ge 2$ the last term of the second sum is smaller than the first term of the first sum etc.).\\
    Therefore we set $\frac{1}{\xi}:=\frac{1}{5\cdot 2^{\varepsilon}}+\frac{4}{5}<1$ and  with $x:= \sum_{j=0}^{\left\lfloor \frac{n}{2} \right\rfloor} \frac{(n-j-1)!}{(n-K-j)!}j$, $y:=\sum_{j=\left\lfloor \frac{n}{2} \right\rfloor+1}^{n-K} \frac{\left(n-j-1 \right)!}{\left(n-K-j \right)!}j$ and some $\eta \ge 0$ we obtain
    \begin{equation*}
        \begin{split}
        \frac{1}{2^{\varepsilon}}x+y & = \frac{1}{2^\varepsilon}\left(\frac{1}{5}(x+y)+\eta \right)+\frac{4}{5}(x+y)-\eta\\
        & \le \left(\frac{1}{5\cdot 2^\varepsilon}+\frac{4}{5} \right)(x+y)\\
        & = \frac{1}{\xi}(x+y).
        \end{split}
    \end{equation*}
The statement now follows with $\sum_{j=0}^{n-K} \frac{(n-j-1)!}{(n-K-j)!}j = \frac{n!}{K(K+1)(n-K-1)!}$.
\end{proof}

For bounding $\ell_p$ distances note that it is possible to define random variables $Y,(Y_n)_{n \ge 1}$ on a common probability space, the so-called optimal couplings, such that
$$\ell_p\left(Y_n,Y\right):=\ell_p\left(\mathcal{L}(Y_n),\mathcal{L}(Y)\right)=\left\Vert Y_n-Y \right\Vert_p.$$ Therefore, for fixed $2 \le K \le 4$, we can define $\kl[\big]{Z_K,\left(Y_n\right)_{n \ge 0}}$ such that 
$$\Delta(n):=\ell_2\left(Y_n,Z_K\right)=\left\Vert Y_n-Z_K \right\Vert_2.$$ 
They are also optimal $\ell_p$-couplings for every $p \ge 3$, see, e.g.,  \cite{vi09}.  Furthermore, we choose $\kl[\big]{Z_K^{(i)},\kl[\big]{Y_n^{(i)}}_{n \ge 0}}$ as independent copies of $\kl[\big]{Z_K,\left(Y_n\right)_{n \ge 0}}$. With the distributional recurrences for $Z_K$ and $Y_n$ we get
\begin{align}{\label{W_i}}
    \begin{split}
       \Delta^2(n) & \le \mathbb{E}\left[\left\vert \sum_{i=0}^K \frac{I_i^{(n)}}{n} Y_{I_i^{(n)}}^{(i)}-D_i Z_K^{(i)} + b_n - b \right\vert^2\right]\\
       & =: \mathbb{E}\left[\left\vert \sum_{i=0}^K W_i + W_{K+1} \right\vert^2\right].
    \end{split}
\end{align}
Conditionally on $I^{(n)}$ and $D$, the terms $W_0,\dots,W_{K+1}$ are independent. Furthermore, we have $\mathbb{E}[W_i]=0$ for $i=0,\dots,K$ and therefore
\begin{align*}
    \mathbb{E}\Bigg[\Bigg(\sum_{i=0}^{K+1} W_i \Bigg)^2 \Bigg|\left(I^{(n)},D \right)\Bigg] = \sum_{i=0}^{K+1} \Ek*{W_i^2 \given \kl[\big]{I^{(n)},D }}
\end{align*}
and with \eqref{W_i} we obtain
\begin{align}\label{Delta_est}
    \Delta^2(n) \le \sum_{i=0}^{K+1} \mathbb{E}\left[W_i^2 \right]=\sum_{i=0}^K \mathbb{E} \Bigg[\Bigg(\frac{I_i^{(n)}}{n} Y_{I_i^{(n)}}^{(i)}-D_i Z_K^{(i)} \Bigg)^2\Bigg] + \mathbb{E}\left[\left(b_n-b \right)^2\right].
\end{align}

Our basic strategy to obtain bounds on $\ell_p$ distances for all $p\ge 2$ is as in \cite{Fill2002} for the case $K=1$, i.e., we argue with induction over $p$ and start with the base case $p=2$.

First, we bound the toll term $b_n-b$ which requires more effort and leads to different bounds in Lemma \ref{toll} compared to the case $K=1$. Recall that
\begin{align*}
b_n&=\frac{1}{n}\left(P_n-\mathbb{E}[X_n]+\sum_{i=0}^K {\mathbb{E}\Big[X_{I_i^{(n)}}\Big| I_i^{(n)}\Big]}\right),\\
b&=\alpha_K \sum_{i=0}^K D_i \ln(D_i) + \sum_{t \in T} \eins_{\left\{D \in C_t\right\}} l_t(D).
\end{align*}
\begin{lem}\label{toll}
    For all $p \ge 1$ and $1 \le K \le 4$ we have
    \begin{align*}
        \left\Vert b_n-b \right\Vert_p = \O\left(\frac{1}{\sqrt{n}}\right).
    \end{align*}
\end{lem}
\begin{proof}
By triangle inequality we have
    \begin{align}\label{eqrn1}
        \begin{split}
           \left\Vert b_n-b \right\Vert_p & \le \left\Vert \frac{P_n}{n} - \sum_{t \in T} \eins_{\left\{D \in C_t\right\}}l_t(D)\right\Vert_p\\
           & \quad\quad~+ \left\Vert \frac{1}{n} \left(\sum_{i=0}^K \mathbb{E}\Big[X_{I_i^{(n)}}\Big|{I_i^{(n)}}\Big]-\mathbb{E}[X_n]\right)-\alpha_K \sum_{i=0}^K D_i \ln\left(D_i\right)\right\Vert_p.
        \end{split}
    \end{align}
    For the first summand in the latter display we obtain
    \begin{align*}
    \begin{split}
    \left\Vert \frac{P_n}{n} - \sum_{t \in T} \eins_{\left\{D \in C_t\right\}}l_t(D)\right\Vert_p \le &\sum_{t \in T} \sum_{i=0}^K \left\vert h_j(t) \right\vert \left\Vert\left (\frac{A_{t,i}^{(n)}}{n}-\eins_{\left\{D \in C_t\right\}}D_i\right)\right\Vert_p+\O\left(\frac{1}{n}\right)
    \end{split}
    \end{align*}
    and similarly to the proof of Theorem~\ref{thm_1} there exists an $n_0 \in \mathbb{N}$ such that
    \begin{align*}
    \begin{split}
    \lnorm*p{\frac{A_{t,i}^{(n)}}{n}-\eins_{\mg*{D \in C_t}}D_i}& =\E{\left\vert\frac{A_{t,i}^{(n)}}{n}-\eins_{\left\{D \in C_t\right\}}D_i \right\vert^p}^{\frac{1}{p}}\\
            & \le \kl[\Bigg]{\P{D \notin C_t} \E{\abs*{\frac{n_0}{n}}^p}
            +\mathbb{E}\left[\abs[\bigg]{\frac{A_{t,i}^{(n)}}{n}-\eins_{\left\{D \in C_t\right\}}D_i }^p \eins_{\left\{D \in C_t\right\}}\right]}^{\frac{1}{p}}\\
            & \le \frac{2n_0}{n}+\lnorm*p{\frac{I_i^{(n)}}{n}-D_i}.
        \end{split}
    \end{align*}
    Let $B_{n,u}$ denote a binomial-$(n,u)$-distributed and $\Ber_u$ a Bernoulli-$u$-distributed random variable.  Further, let $B(p,q)$ denote the beta-function with parameters $p$ and $q$. In particular, we have $B(1,K)=\frac{1}{K}$. Using bounding ideas of \cite{Neininger2015}, we condition on $D_i$, which is beta$(1,K)$-distributed, and obtain  
    \begin{align*}
        \begin{split}
          \lnorml*p{\frac{I_i^{(n)}}{n-K}-D_i}
            & = \frac{1}{B(1,K)}\int_0^1 {\left(1-u\right)^{K-1}} \mathbb{E}\Bigg[\bigg| \frac{I_i^{(n)}}{{n-K}}-D_i \bigg|^p \;\Bigg| \;D_i=u \Bigg]\mathrm du\\
            & = K\int_0^1 \left(1-u \right)^{K-1} \frac{1}{\left(n-K\right)^p} \Ek[\big]{\left\vert B_{n-K,u}-(n-K)u \right\vert^p }\mathrm du .
        \end{split}
    \end{align*}
    We now use the Marcinkiewicz--Zygmund inequality \cite{Chow1988} to get
    \begin{align*}
    \begin{split} 
    \lnorm*p{\frac{I_i^{(n)}}{n-K}-D_i}
            &\le \Bigg(K\int_0^1 \left(1-u \right)^{K-1} \frac{M_p}{(n-K)^p} \mathbb{E}\Bigg[\Bigg(\sum_{j=1}^{n-K}\left| \Ber_u \right|^2\Bigg)^{\frac{p}{2}}\Bigg]\mathrm du\Bigg)^{\frac{1}{p}}\\
            &\le \left(K\int_0^1 (1-u)^{K-1} \frac{M_p}{\left(n-K \right)^{\frac{p}{2}}} \mathrm du\right)^{\frac{1}{p}}\\
            &= \frac{M_p^{1/p}}{\sqrt{n-K}}
            \end{split}
   \end{align*}
    with a constant $M_p$ which only depends on $p$. Overall we obtain
    \begin{align}\label{boundI_i}
        \left\Vert \frac{I_i^{(n)}}{n}-D_i\right\Vert_p\le \left\Vert \frac{I_i^{(n)}}{n}-\frac{I_i^{(n)}}{n-K} \right\Vert_p+\left\Vert \frac{I_i^{(n)}}{n-K}-D_i \right\Vert_p=\O\left(\frac{1}{\sqrt{n}}\right)
    \end{align}
    and hence receive our bound for the first summand in (\ref{eqrn1})
    \[
    \left\Vert \frac{P_n}{n} - \sum_{t \in T} \eins_{\left\{D \in C_t\right\}}l_t(D)\right\Vert_p=\O\left(\frac{1}{\sqrt{n}}\right).
    \]
    To bound the second summand in (\ref{eqrn1}) tightly we need to improve on the bounds used to prove \Cref{thm:moment}. For $K≤4$, see \eqref{eq:error-bounds}, we have
    \begin{align*}
        \MoveEqLeft[1]
        \lnorm*p{
                \frac{1}{n} \kl*{\sum_{i=0}^K \E{X_{I_i^{(n)}}\mid{I_i^{(n)}}}-\E{X_n}}
               -\alpha_K \sum_{i=0}^K D_i \ln(D_i)
        }
        \\&= \Bigg\|
            \frac{1}{n}\left(\sum_{i=0}^K \alpha_K I_i^{(n)}\ln\left(I_i^{(n)}\right)+\beta_K I_i^{(n)}-\left(\alpha_K n\ln(n)+\beta_K n\right)\right)
                     \\&\quad \quad\quad~  - \sum_{i=0}^{K} \alpha_K D_i \ln(D_i) \Bigg\|_p+\O\left(\frac{1}{\sqrt{n}}\right)
        \\&= \lnorm*p{
            - \frac{\beta_KK}{n}+\alpha_K\left(\sum_{i=0}^K \frac{I_i^{(n)}}{n}\ln\frac{I_i^{(n)}}{n}-D_i\ln(D_i)\right)-\alpha_K \frac{K}{n} \ln(n)
            }
            +\O\left(\frac{1}{\sqrt{n}}\right)
        \\&≤ \alpha_K \sum_{i=0}^K \lnorm*p{\frac{I_i^{(n)}}{n} \ln\frac{I_i^{(n)}}{n}-D_i \ln(D_i)}+\O\left(\frac{1}{\sqrt{n}}\right).
     \end{align*}
     Using the same arguments as in the proof of \cite[Proposition 2.2]{Neininger2015} for $p=3$, we obtain
     \begin{align*}
     \left\Vert \frac{I_i^{(n)}}{n} \ln\frac{I_i^{(n)}}{n}-D_i \ln(D_i)\right\Vert_p =\O\left(\frac{1}{\sqrt{n}}\right),
     \end{align*}
     hence the statement of Lemma \ref{toll} follows.
\end{proof}
 Let $\varepsilon>0$ be fixed. We are proving 
 \begin{align} \label{l2bound}
     \Delta(n) \le cn^{-1/2+\varepsilon}
 \end{align}
 for an appropriate constant $c>0$ by induction over $n$. The induction start is clear. Recall that we have the bound (\ref{Delta_est}) for $\Delta(n)$. To bound  the first summand on the right hand side of (\ref{Delta_est}) we start rewriting
\begin{align*}
    \begin{split}
        \left(\frac{I_i^{(n)}}{n} Y_{I_i^{(n)}}^{(i)}-D_i Z_K^{(i)}\right)^2 
        & = \left(\frac{I_i^{(n)}}{n} \left(Y_{I_i^{(n)}}^{(i)}-Z_K^{(i)}\right)\right)^2+\left(\left(\frac{I_i^{(n)}}{n}-D_i \right)Z_K^{(i)}\right)^2\\
        & \quad~\quad\quad + 2 \frac{I_i^{(n)}}{n} \left(Y_{I_i^{(n)}}^{(i)}-Z_K^{(i)}\right)\left(\frac{I_i^{(n)}}{n}-D_i \right)Z_K^{(i)}.
    \end{split}
\end{align*}
For the final factor in the latter display we have
\begin{equation*}
   \mathbb{E}\left[\bigg(\bigg(\frac{I_i^{(n)}}{n}-D_i \bigg)Z_K^{(i)}\bigg)^2 \right] \le \bigg\| \frac{I_i^{(n)}}{n}-D_i \bigg\|_2^2 \Vert Z_K \Vert_2^2 = \O\left(\frac{1}{n} \right),
\end{equation*}
since $Z_K$ has a finite second moment, see Theorem~\ref{thm_1}. Conditioning on  $I_i^{(n)}=j$ and $D_i=u$, we have
\begin{equation*}
    \begin{split}
        \mathbb{E}\left[\frac{j}{n}\left(\frac{j}{n}-u \right) \left(Y_{j}^{(i)}-Y^{(i)}\right)Z_K^{(i)}\right]
        & \le \frac{j}{n} \left\vert \frac{j}{n}-u \right\vert \left\Vert Y_j-Z_K \right\Vert_2 \left\Vert Z_K \right\Vert_2\\
        & =\frac{j}{n} \left\vert \frac{j}{n}-u \right\vert \Delta(j) \sigma,
    \end{split}
\end{equation*}
using the Cauchy--Schwarz inequality and $\sigma:=\Vert Z_K \Vert_2<\infty.$ With the inductive hypothesis $\Delta(j) \le cj^{-1/2+\varepsilon}$ for $j<n$ we obtain
\begin{align*}
\begin{split}
        \Ek*{\frac{I_i^{(n)}}{n} \left(Y_{I_i^{(n)}}^{(i)}-Z_K^{(i)}\right)\left(\frac{I_i^{(n)}}{n}-D_i \right)Z_K^{(i)} \given I_i^{(n)}=j,D_i=u }
        &\le c\sigma\frac{j^{\frac{1}{2}+\varepsilon}}{n}\left\vert \frac{j}{n}-u \right\vert\\
        &\le \frac{c \sigma}{n^{\frac{1}{2}-\varepsilon}}\left\vert \frac{j}{n}-u \right\vert.
    \end{split}
\end{align*}
and there exists an $w>0$ such that
\begin{align*}
\begin{split}
       \E{\frac{I_i^{(n)}}{n} \left(Y_{I_i^{(n)}}^{(i)}-Z_K^{(i)}\right) \left(\frac{I_i^{(n)}}{n}-D_i \right)Z_K^{(i)}}
       \nonumber & = \E{\Ek[\bigg]{\frac{I_i^{(n)}}{n} \left(Y_{I_i^{(n)}}^{(i)}-Z_K^{(i)}\right)\kl[\bigg]{\frac{I_i^{(n)}}{n}-D_i}Z_K^{(i)} \given I_i^{(n)},D_i}}\nonumber\\
       & \le \frac{c\sigma}{n^{\frac{1}{2}-\varepsilon}} \left\Vert \frac{I_i^{(n)}}{n}-D_i \right\Vert_1\nonumber\\
       &\le \frac{w\sigma c}{n^{\frac{1}{2}-\varepsilon}}\frac{1}{\sqrt{n}}= \frac{w\sigma c}{n^{1-\varepsilon}}
       \end{split}
\end{align*}
Note that the use of optimal couplings implies 
$\mathbb{E}[(j/n)^2 (Y_j-Z_K )^2]=(j/n)^2 \Delta^2(j)$, hence
\[
\mathbb{E}\Bigg[\Bigg(\frac{I_i^{(n)}}{n} \Bigg(Y_{I_i^{(n)}}^{(i)}-Z_K^{(i)}\Bigg)\Bigg)\Bigg]^2= \mathbb{E}\Bigg[\Bigg(\frac{I_i^{(n)}}{n}\Bigg)^2\Delta^2\Big(I_i^{(n)}\Big)\Bigg].
\]
Collecting our estimates, we obtain 
\[
\Delta^2(n) \le (K+1) \Bigg(\mathbb{E}\Bigg[\Bigg(\frac{I_0^{(n)}}{n}\Bigg)^2 \Delta^2\left(I_0^{(n)}\right)\Bigg]+\frac{w \sigma c}{n^{1-\varepsilon}}\Bigg)+\O\left(\frac{1}{n}\right),
\]
since the random variables $I_i^{(n)}$ are identically distributed for all $i=0,\dots,K$. Note that $D_i$ is $\operatorname{Beta}(1,K)$ distributed and given $D=\left(d_0,\dots,d_K\right)$ the sizes of the subproblems $I^{(n)}$ are multinominally$(n-K;d_0,\dots,d_K)$ distributed. Using the inductive hypothesis, we have
\begin{equation*}
    \begin{split}
        \mathbb{E} \left[\left(\frac{I_0^{(n)}}{n}\right)^2 \Delta^2\left(I_0^{(n)}\right)\right]
        & = \frac{1}{B(1,K)}\int_0^1 \left(1-u \right)^{K-1} \sum_{j=0}^{n-K} \binom{n-K}{j} u^j (1-u)^{n-K-j}\\
        & \quad \quad \quad \quad \quad \quad~ \times\mathbb{E}\left[\left(\frac{I_0^{(n)}}{n}\right)^2 \Delta^2\left(I_0^{(n)}\right) \mid \left(I_0^{(n)}=j,D_0=u \right)\right] \mathrm du\\
        & \le \frac{K}{n^2} \sum_{j=1}^{n-K} \frac{c^2(n-K)!}{j! (n-K-j)!} \frac{j^2}{j^{1-2\varepsilon}} \int_0^1 u^j (1-u)^{n-j-1} \mathrm du\\
        & = \frac{c^2K(n-K)!}{n^2}\sum_{j=1}^{n-K} \frac{1}{j! (n-K-j)!} \frac{j! (n-j-1)!}{n!}j^{1+2\varepsilon}\\
        & = \frac{c^2K(n-K)!}{n^2 n!} \sum_{j=1}^{n-K} \frac{(n-j-1)!}{(n-K-j)!} j^{1+2\varepsilon}.
    \end{split}
\end{equation*}
By Lemma \ref{tech} there exists a $\xi>1$ such that
\[
\sum_{j=1}^{n-K} \frac{(n-j-1)!}{(n-K-j)!} j^{1+2\varepsilon} \le \frac{n^{2\varepsilon}}{\xi} \frac{n!}{K(K+1)(n-K-1)!}.
\]
It follows
\[
\E{\Bigg(\frac{I_0^{(n)}}{n}\Bigg)^2 \Delta^2\left(I_0^{(n)}\right)}≤ \frac{c^2(n-K)n^{2\varepsilon}}{\xi (K+1) n^2} \le \frac{c^2}{\xi (K+1)} \frac{1}{n^{1-2\varepsilon}}
\]
and putting the estimates together, we obtain with an appropriate constant $d>0$ that
\begin{equation*}
    \begin{split}
         \Delta^2(n) & \le (K+1)\left(\frac{c^2} {\xi (K+1)} \frac{1}{n^{1-2\varepsilon}}+\frac{w\sigma c}{n^{1-2\varepsilon}}\right)+\frac{d}{n^{1-2\varepsilon}}\\
         & = \left(\frac{1}{\xi}c^2+(K+1)w\sigma c+d\right) \frac{1}{n^{1-2\varepsilon}}\\
         & \le c^2 \frac{1}{n^{1-2\varepsilon}},
    \end{split}
\end{equation*}
the last inequality being valid for sufficiently large $c$ in view of $\frac{1}{\xi}<1$. This finishes the proof of the bound on the $\ell_2$ rate of convergence stated in (\ref{l2bound}).\\
We now extend the bound in (\ref{l2bound}) to $\ell_p$ for every $p \ge 1$. Because $\ell_p \le \ell_q$ for $p \le q$, it is sufficient to consider only $p \in \mathbb{N}$. The case $p=2$ has just been shown above. We now consider $p \in \mathbb{N}$ with $p \ge 3$. Similar to Lemma 3.2 in \cite{Fill2002}, we have for every $m \in \mathbb{N}$, independent random variables $Q_1,\dots,Q_{m+1}$ and $p \in \mathbb{N}, p \ge 2$ that
\begin{align}\label{independent}
\E{\abs[\bigg]{\sum_{i=1}^{m+1} Q_i }^p} \le \sum_{i=1}^m \mathbb{E}\left[\left\vert Q_i \right\vert^p \right]+\left(\sum_{i=1}^m \left\Vert Q_i \right\Vert_{p-1}+\left\Vert Q_{m+1} \right\Vert_p \right)^p.
\end{align}
We obtain
\[
\Delta_p(n):=\ell_p\left(Y_n,Z_K\right) ≤\lnorm*p{\sum_{i=0}^{K+1} {W_i}}
\]
with the $W_i$ defined in (\ref{W_i}). The Minkowski inequality yields
\begin{align*}
    \mathbb{E}\Big[\left\vert W_i \right\vert ^{p-1} \;\Big|\; I_i^{(n)}=j,D_i=u \Big]^{\frac{1}{p-1}} \le \frac{j}{n} \ell_{p-1}\left(Y_j,Z_K \right)+\left\vert \frac{j}{n}-u \right\vert \left\Vert Z_K \right\Vert_{p-1}.
\end{align*}
The second part of Theorem~\ref{thm_1} implies $\tau:= \left\Vert Z_K\right\Vert_{p-1}<\infty$ since a finite moment generating function yields $\Vert Z_K \Vert_p<\infty$ for all $p$. The inductive hypothesis for induction on $p$ is $\ell_{p-1}\left(Y_j,Z_K\right) \le c_{p-1} j^{-\frac{1}{2}+\varepsilon}$ for all $j\ge 1$. Hence, we obtain
\[
\mathbb{E}\left[\left\vert W_i \right\vert ^{p-1} \mid \left(I_i^{(n)}=j,D_i=u\right)\right]^{\frac{1}{p-1}} \le c_{p-1}\frac{j^{\frac{1}{2}+\varepsilon}}{n}+\tau \left\vert \frac{j}{n}-u \right\vert \le \frac{c_{p-1}}{n^{\frac{1}{2}-\varepsilon}}+\tau \left\vert \frac{j}{n}-u \right\vert
\]
and therefore, writing $\Ekv D{\cdot} := \Ek{\cdot\given D, I^{(n)}}$,
\begin{align*}
\Ekv*D{\abs*{W_i }^{p-1}}^{\frac{1}{p-1}} \le \frac{c_{p-1}}{n^{\frac{1}{2}-\varepsilon}}+\tau \left\vert \frac{I_i^{(n)}}{n}-D_i \right\vert.
\end{align*}
Let $k=(k_0, k_1, \dots, k_{K+1}) \in \N^{K+2}$ be a multiindex with $\abs k := k_0+\dots+k_{K+1}=p$.
Expanding the power of the sum,
\begin{align*}
    \begin{split}
    \kl*{\sum_{i=0}^K \Ekv*D{\abs{W_i}^{p-1}}^{\frac{1}{p-1}}+\Ekv*D{\left\vert b_n - b \right\vert^p}^{\frac{1}{p}}}^p
    &=\sum_{\abs k=p} \binom{p}{k}
    \prod_{i=0}^K\left(\Ekv*D{\left\vert W_i \right\vert^{p-1}}^{\frac{1}{p-1}}\right)^{k_i}
    \Ekv*D{\abs{b_n - b}^p}^{\frac{k_{K+1}}{p}}\\
    &≤\sum_{\abs k=p} \binom{p}{k}\prod_{i=0}^{K} \left(\frac{c_{p-1}}{n^{\frac{1}{2}-\varepsilon}}+\tau \abs[\bigg]{ \frac{I_i^{(n)}}{n}-D_i} \right)^{k_i}
    \hspace{-1em}
    \Ekv*D{\abs{b_n - b}^p}^{\frac{k_{K+1}}{p}},
    \end{split}
\end{align*}
and with \eqref{independent} we have
\begin{equation*}
        \E{\abs[\bigg]{\sum_{i=0}^{K+1} W_i }^p} ≤ 
        \sum_{i=0}^K \E{\abs{W_i}^p}+\sum_{\abs k=p} \binom{p}{k}
        \Ek[\Bigg]{\prod_{i=0}^{K} \left(\frac{c_{p-1}}{n^{\frac{1}{2}-\varepsilon}}+\tau \left\vert \frac{I_i^{(n)}}{n}-D_i \right\vert \right)^{k_i}
        \hspace{-1em}\Ekv*D{\abs{b_n - b}^p}^{\frac{k_{K+1}}{p}}}.
\end{equation*}
To further analyze the latter term we use the H\"older's inequality, which implies
\begin{equation}\label{Hoelder}
    \begin{split}
        &\mathbb{E}\left[\prod_{i=0}^{K} \left(\frac{c_{p-1}}{n^{\frac{1}{2}-\varepsilon}}+\tau \left\vert \frac{I_i^{(n)}}{n}-D_i \right\vert \right)^{k_i}
        \Ekv*D{\abs{b_n - b}^p}^{\frac{k_{K+1}}{p}}\right]\\
        & \quad \le \prod_{i=0}^K
        \lnorm[\Bigg]{K+2}{ 
            \left(\frac{c_{p-1}}{n^{\frac{1}{2}-\varepsilon}}+\tau \abs[\bigg]{\frac{I_i^{(n)}}{n}-D_i} \right)^{k_i}
        }
        \lnorm[\Bigg]{K+2}{
        \Ekv*D{\abs{b_n - b}^p}^{\frac{k_{K+1}}{p}}
        }\hspace{-1.5em}.
    \end{split}
\end{equation}
The second factor of the latter term is bounded with Lemma \ref{toll} by
\[
\lnorm[\Bigg]{\mathrlap{K+2}}{
        \Ekv*D{\abs{b_n - b}^p}^{\frac{k_{K+1}}{p}}
}\;\;=\O\left({n^{-k_{K+1}/2}}\right).
\]
For the first term in \eqref{Hoelder}, we have for $k_i \ge 1$
\begin{align*}
    \begin{split}
        \mathbb{E}\left[\left\vert \frac{c_{p-1}}{n^{\frac{1}{2}-\varepsilon}}+\tau \left\vert\frac{I_i^{(n)}}{n}-D_i\right\vert\right\vert^{(K+2)k_i}\right]
        & =\sum_{j=0}^{(K+2)k_i} \binom{(K+2)k_i}{j}\frac{c_{p-1}^j}{n^{\frac{j}{2}-j\varepsilon}}\tau^{(K+2)k_i-j}\mathbb{E}\left[\left\vert \frac{I_i^{(n)}}{n}-D_i \right\vert ^{(K+2)k_i-j}\right]\\
        & =\sum_{j=0}^{(K+2)k_i} \binom{(K+2)k_i}{j} \frac{c_{p-1}^j}{n^{\frac{j}{2}-j\varepsilon}}\tau^{(K+2)k_i-j}\left\Vert \frac{I_i^{(n)}}{n}-D_i \right\Vert_{(K+2)k_i-j}^{(K+2)k_i-j}.
    \end{split}
\end{align*}
With \eqref{boundI_i} there exists some constant $w_i>0$ such that
\begin{align*}
\begin{split}
       \mathbb{E}\left[\left\vert \frac{c_{p-1}}{n^{\frac{1}{2}-\varepsilon}}+\tau \left\vert\frac{I_i^{(n)}}{n}-D_i\right\vert\right\vert^{(K+2)k_i}\right] & \le \sum_{j=0}^{(K+2)k_i} \binom{(K+2)k_i}{j} \frac{c_{p-1}^j}{n^{\frac{j}{2}-j\varepsilon}}\tau^{(K+2)k_i-j}\frac{w_i}{n^{\frac{(K+2)k_i-j}{2}}}\\
       & \le \sum_{j=0}^{(K+2)k_i} \binom{(K+2)k_i}{j} w_i c_{p-1}^j \tau^{(K+2)k_i-j}\frac{1}{n^{\frac{(K+2)k_i}{2}-(K+2)k_i\varepsilon}}\\
       & =\O\left(\frac{1}{n^{\frac{(K+2)k_i}{2}-(K+2)k_i\varepsilon}}\right).
\end{split}
\end{align*}
This yields 
\begin{align*}
\left\Vert \left(\frac{c_{p-1}}{n^{\frac{1}{2}-\varepsilon}}+\tau \left\vert \frac{I_i^{(n)}}{n}-D_i \right\vert\right)^{k_i} \right\Vert_{K+2}=\O\left(\frac{1}{n^{k_i/2-k_i\varepsilon}}\right)
\end{align*} 
and overall we obtain
\[
\mathbb{E}\left[\prod_{i=0}^{K} \left(\frac{c_{p-1}}{n^{\frac{1}{2}-\varepsilon}}+\tau \left\vert \frac{I_i^{(n)}}{n}-D_i \right\vert \right)^{k_i}
\Ekv*D{\abs{b_n - b}^p}^{\frac{k_{K+1}}{p}}\right]=\O\left(\frac{1}{n^{\sum k_i/2-\sum k_i\varepsilon}}\right)
\]
and thus
\begin{align*}
    \begin{split}
        \sum_{\abs k=p} \binom{p}{k}
        \Ek[\Bigg]{\prod_{i=0}^{K} \left(\frac{c_{p-1}}{n^{\frac{1}{2}-\varepsilon}}+\tau \left\vert \frac{I_i^{(n)}}{n}-D_i \right\vert \right)^{k_i}
        \hspace{-1em}\Ekv*D{\abs{b_n - b}^p}^{\frac{k_{K+1}}{p}}}=\O\left({n^{-\frac{p}{2}+p\varepsilon}}\right).
    \end{split}
\end{align*}
Collecting the estimates, we obtain
\begin{align*}
    \Delta_p^p(n) \le (K+1) \mathbb{E}\left[\left\vert W_0 \right\vert^p \right]+\O\left(\frac{1}{n^{p\left(\frac{1}{2}-\varepsilon\right)}}\right).
\end{align*}
For the term $\mathbb{E}\left[\left\vert W_0 \right\vert^p \right]$, analogously to the case $p=2$, we have
\begin{align*}
    \begin{split}
        \E{\left\vert W_0 \right\vert^p }& = \mathbb{E}\left[\left\vert W_0 \right\vert^p \mid \left(I_0^{(n)},D_0 \right)\right]\\
        & = \sum_{r=0}^{p-1} \binom{p}{r}\mathbb{E}\left[\left(\frac{I_0^{(n)}}{n}\right)^r \Delta_p^r\left(I_0^{(n)}\right)\tau^{p-r}\left\vert \frac{I_0^{(n)}}{n}-D_0 \right\vert^{p-r}\right]\\
        & \quad + \mathbb{E}\left[\left(\frac{I_0^{(n)}}{n}\right)^p \Delta_p^p\left(I_0^{(n)}\right)\right].
\end{split}
\end{align*}
The inductive hypothesis $\Delta_p(j) \le cj^{-\frac{1}{2}+\varepsilon}$ for $j<n$ yields
\begin{align*}
    \begin{split}
        \mathbb{E}\left[\left(\frac{I_0^{(n)}}{n} \right)^r \Delta_p^r\left(I_0^{(n)}\right)\tau^{p-r}\left\vert \frac{I_0^{(n)}}{n}-D_0 \right\vert^{p-r}\right] & \le \frac{\tau^{p-r}c^r}{n^{r/2-r\varepsilon}} \mathbb{E}\left[\left\vert \frac{I_0^{(n)}}{n}-D_0 \right\vert^{p-r}\right]\\
        & \le \frac{a_r c^r}{n^{p/2-p\varepsilon}}
    \end{split}
\end{align*}
for some constants $a_r>0$ for $r=0,\dots,p-1$. The term $\mathbb{E}\left[\left(\frac{I_0^{(n)}}{n}\right)^p \Delta_p^p\left(I_0^{(n)}\right)\right]$ is bounded explicitly through
\begin{align*}
    \begin{split}
        \mathbb{E} \left[\left(\frac{I_0^{(n)}}{n}\right)^p\Delta_p^p\left(I_0^{(n)}\right)\right]
        & = \frac{1}{B(1,K)}\int_0^1 \left(1-u \right)^{K-1} \sum_{j=0}^{n-K} \binom{n-K}{j} u^j (1-u)^{n-K-j}\\
        & \quad \quad \quad \quad \quad \quad \times\mathbb{E}\left[\left(\frac{I_0^{(n)}}{n}\right)^p \Delta_p^p\left(I_0^{(n)}\right) \mid \left(I_0^{(n)}=j,D_0=u \right)\right] du\\
        & \le \frac{K}{n^p} \sum_{j=1}^{n-K} \frac{c^p(n-K)!}{j! (n-K-j)!} \frac{j^p}{j^{\frac{p}{2}-p\varepsilon}} \int_0^1 u^j (1-u)^{n-j-1} du\\
        & \le \frac{c^p K(n-K)! n^{\frac{p}{2}-1}}{n^p}\sum_{j=1}^{n-K} \frac{1}{j! (n-K-j)!} \frac{j! (n-j-1)!}{n!}j^{1+p\varepsilon}\\
        & = \frac{c^p K(n-K)!}{n^{\frac{p}{2}+1} n!} \sum_{j=1}^{n-K} \frac{(n-j-1)!}{(n-K-j)!} j^{1+p\varepsilon}.
    \end{split}
\end{align*}
By Lemma \ref{tech} there exists a $\xi>1$ (depending in $p$ and being different from the $\xi$ appearing above) such that
\begin{align*}
    \sum_{j=1}^{n-K} \frac{(n-j-1)!}{(n-K-j)!} j^{1+p\varepsilon} \le \frac{n^{p\varepsilon}}{\xi} \frac{n!}{K(K+1)(n-K-1)!}.
\end{align*}
Plugging in, we obtain 
\begin{align*}
    \mathbb{E}\left[\left(\frac{I_0^{(n)}}{n}\right)^p \Delta^p\left(I_0^{(n)}\right)\right] & \le \frac{c^p(n-K)n^{p\varepsilon}}{\xi (K+1) n^{\frac{p}{2}+1}} \le \frac{c^p}{\xi (K+1)} \frac{1}{n^{\frac{p}{2}-p\varepsilon}}.
\end{align*}
Overall we have
\begin{align*}
    \Delta_p^p(n) \le \left(\frac{1}{\xi}c^p+ \sum_{i=0}^{p-1} \tilde{a}_i c^i\right)\frac{1}{n^{\frac{p}{2}-p\varepsilon}} \le \frac{c^p}{n^{\frac{p}{2}-p\varepsilon}}
\end{align*}
with some constants $\tilde{a}_0,\tilde{a}_1,\dots,\tilde{a}_{p-1}>0$ and $c$ sufficiently large. This finishes the proof of the bounds on the $\ell_p$ metrics.

To bound the distance between $Y_n$ and $Z_K$ in the Kolmogorov-Smirnov metric, we use Lemma 5.1 in \cite{Fill2002}, which implies
\[
\varrho\left(Y_n,Z_K \right) \le \left(\left(1+p\right) \left\Vert f_{Z_K} \right\Vert^p_{\infty}\right)^{\frac{1}{p+1}} \left(\ell_p\left(Y_n,Z_K \right)\right)^{\frac{p}{p+1}}
\]
since $Z_K$ has a bounded density function $f_{Z_K}$ with Theorem~\ref{thm:smooth}. We know that for all $p \ge 1$ and $\delta>0$
\[
\ell_p\left(Y_n,Z_K\right) \le \frac{c_p}{n^{\frac{1}{2}-\delta}}
\]
with some constant $c_p$. For some fixed $\varepsilon$, we can choose $p$ large enough such that $p/(2(1+p))>\frac{1}{2}-\varepsilon$. It is possible to choose $\delta>0$ with $(p/(1+p))\left(\frac{1}{2}-\delta\right)>\frac{1}{2}-\varepsilon$ and thereby obtain
\[
\varrho\left(Y_n,Z_K \right) \le c_p'\frac{1}{n^{\frac{p}{1+p}\left(\frac{1}{2}-\delta\right)}} \le c_p' \frac{1}{n^{\frac{1}{2}-\varepsilon}}
\]
where the constant $c_p'$ depends on $\varepsilon$ but not on $n$. This finishes the proof of Theorem~\ref{thm:speed}.

\bibliographystyle{amsplain}
\bibliography{patricia_bib}

\end{document}